\documentclass[a4paper, 12pt]{article}

\usepackage{amsthm, amsmath, amssymb, mathrsfs, multirow, url, subfigure}
\usepackage{graphicx}
\usepackage{ifthen}
\usepackage{amsfonts}
\usepackage[usenames]{color}
\usepackage{fullpage}
\usepackage{hhline}
\usepackage{cite}

\usepackage[colorlinks,citecolor=blue]{hyperref}

\theoremstyle{plain}
\newtheorem{theorem}{Theorem}
\newtheorem{corollary}{Corollary}
\newtheorem{proposition}{Proposition}
\newtheorem{lemma}{Lemma}

\theoremstyle{definition}
\newtheorem{defn}{Definition}

\theoremstyle{remark}
\newtheorem{remark}{Remark}

\renewcommand{\phi}{\varphi}

\title{On the Nature of Saturated $2^k$- Factorial Designs for Unbiased Estimation of Non-negligible Parameters}
\author{Francois Domagni \qquad A. S.  Hedayat \qquad  Bikas Kumar Sinha\\
Department of Mathematics, Statistics, and Computer Science \\
University of Illinois at Chicago \\
{\tt fdomag2@uic.edu, hedayat@uic.edu, bikassinha1946@gmail.com }
}
\date{\today}

\begin{document}

\maketitle

\begin{abstract}
We contemplate an experimental situation in a $2^k$-factorial experiment with acute resource crunch so that we need to conduct just a saturated design [SD] - with the understanding that precision of the estimates cannot be estimated from the data.  It is known beforehand which effect(s)/interaction(s) are likely to be negligible. We examine the flexibility to the extent that an experimenter can make a choice of an SD in order to retain information on all the remaining [non-negligible] effects/interactions.

\smallskip

\end{abstract}

\section{Introduction}
\label{S:intro}

Two-level factorial designs (TLFD) are widely used in scientific and industrial experimentation for various reasons. Standard text books deal with this topic at various lengths - covering such concepts as (i) Unreplicated Full Factorials, (ii) Replicated Full Factorials, (iii) Blocking, (iv) Total, Partial and Balanced/Unbalanced Confounding, (v) Fractional Factorials etc. Practitioners primarily use TLFD at an early stage of an experimentation to screen potential factors that are involved in the system being investigated. The statistical models underlying TLFD are simple and subject to relatively weak assumptions. Each factor - whether quantitative or qualitative - is assumed to have two levels that are conveniently coded as $-1$ or $1$ in the design matrix. The estimators of the effects/interactions  are contrasts that are naturally simple to interpret. The effect of a factor is interpreted as a measure of the change in the response variable due to variation of the factor from low to high - averaged over all other factor levels. \\

In practice investigators postulate, for one reason or the other, that certain effects (usually higher order interactions) are unimportant or negligible. When that is the case it is desirable for them to conduct the experiment with the least number of runs that would ensure the  unbiased estimation of the important effects that is, non-negligible effects of interest. Regular Fractional Factorial Designs (RFFD) are used in this kind of situation and there is a vast literature available on RFFD. See, for example,  Montgomery [\cite{Montgomery}]. 
  
In the framework of a two-level factorial design, one of the drawbacks of RFFD is that the number of runs needed to conduct the experiment is necessarily a multiple of $4$. Thus when the important effects to be estimated are identified beforehand, using an RFFD may lead to the use of more resources than the bare minimum needed for the estimation of the important effects . For instance if the number of factors is $k=5$ and the only important effects are the main effects plus the mean then using an $2^{5-2}_{III}$ RFFD of Resolution $III$ would require $8$ runs for the experiment. This would actually estimate the $5$ main effects plus the mean but also can provide estimate two other effects that are known to be negligible. \\

A Saturated Designs (SD) could be used in case of scarce resources when it is clear to the investigator which effects are important and non-negligible. However it turns out that the identification of an SD is not a trivial problem. Numerous papers available in the literature discuss how to construct SDs under certain conditions. See  Hedayat and Pesotan [\cite{Hedayat1}] and  [\cite{Hedayat2}]. In addition various computer algorithms have been developed to search for SDs in the TLFD set-up. Some of these are SPAN, DETMAX. See Hedayat and  Haiyuan Zhu [\cite{Hedayat3}]. It is worth pointing out that when RFFD are used to estimate a certain vector parameter of interest, the estimator of each effect except the mean is a contrast in terms of the runs and it is clear to practitioners that each estimator measures an interaction or the change in the response variable due to the variation of some factor from low to high. The common practice available in the literature is to choose an SD for which the underlying design matrix is non-singular. The Ordinary Least Squares (OLS) method is then used to obtain the estimator of  the  vector parameter of interest. \\

The question we may ask is the following  " \textit{is the estimator [blue] of each parameter (except the common mean)  in a SD model a contrast in terms of the runs?}". Well, if the design matrix is a Hadamard matrix then the answer is trivially 'yes' since the SD in that case can be seen as an RFFD. However when the design matrix is not a Hadamard matrix the estimator of the vector parameter is given by $\hat{\beta_1} = (D^TD)^{-1}D^TY = D^{-1}Y$ where $D$ is the saturated design matrix. In practice it is desirable to practitioners to have the estimator of each estimable effect as contrast in terms of the runs for the sake of interpretation. It is interesting to verify that it is indeed so even when the design is not based on a Hadamard matrix. This can be seen as follows. \\

Let $D$ be a square non-singular matrix with its first column as a vector of $1$'s. Let $e_1$ denote the column vector of the same dimension with elements $(1, 0, 0, \ldots, 0)$. Then $De_1$ is a vector of $1$'s. Hence, whenever $D$ is non-singular, $D^{-1}$ times vector of $1$'s= $e_1$. This is equivalent to the statement that the estimates of all model parameters [except the common mean ] are linear observational contrasts, irrespective of the nature of elements of the matrix $D$.   \\ 

The rest of the paper is organized as follows. In Section 2, we develop general theory for 'deletion of exact number of runs' so as to ensure estimability of all non-negligible effects/interactions in a saturated $2^k$-factorial experiment. This is done through identification of the runs to be deleted, for any given collection of non-negligible effects/interactions. The choice is not unique. However, we do not necessarily address the question of 'optimal choice'. In Section 3, we take up some illustrative examples in the case of $2^3$- and $2^4$-factorial experiments.  

\section{General theory for identification of runs for deletion - retaining estimability of non-negligible effects/interactions in a saturated design}

We begin by listing several properties of a Hadamard matrix of order $N=2^k$. 

\begin{theorem}
\label{theorem1}
Let $H_N$ be a Hadamard matrix of order $N$ that is partitioned into block matrices as :
\\
$H_N = \begin{bmatrix}
D & E \\
V & C
\end{bmatrix} 
$ where $D$ and $C$ are square matrices of order $n$ and $d$ respectively such that $n\geq d$. Then we have the following results: 
\\
\begin{enumerate}
\item $\boldsymbol{|det(D)| = N^{\frac{n-d}{2}}|det(C)|}$
\item if $D$ is invertible then $C$ is also invertible and the inverse of $D$  is given by:
\\
 $\boldsymbol{  D^{-1}= \frac{1}{N}[D-EC^{-1}V]^T }$
 \end{enumerate}
\end{theorem}
\begin{proof}
We have  $H_N$ is a Hadamard matrix implies that  $H_NH_N^T = H_N^TH_N = NI_N$. Therefore we have :

  $  D^TD = -V^TV + (n + d)I_n $
  
   $  D^TD -\lambda I_n = -V^TV + (n + d)I_n -\lambda I_n$
   
    $  D^TD -\lambda I_n = -V^TV + (n + d)I_n -\lambda I_n$
    
     $  D^TD -\lambda I_n =(n+d - \lambda ) I_n -V^TV $
     \\
     Let   $\gamma = n+d -\lambda $ then we have the following equation:
     
     \begin{equation}
     \label{eigen}
       D^TD -\lambda I_n =\gamma I_n  -V^TV 
     \end{equation}
     
 From Equation (\ref{eigen}) if $\gamma$ is an eigenvalue of $V^TV$ then $\lambda = (n +d)-\gamma $ is eigenvalue for $D^TD$ and   vice versa. 
     \\
     \\
 Now assume $rank(V) = r$. Then    $ r \leq d$ the  $n\times n$ matrix  $V^TV$ has $r$  non-zero eigenvalues of $ \gamma_1,\cdots ,\gamma_r$  and $n-r$ zero eigenvalues. We deduce that $D^TD$ has $n-r$ eigenvalues $\lambda_i = n+d = N$; $ i = 1, \cdots , n-r$. The remaining $r$ eigenvalues of $D^TD$ are
 \\
 $N- \gamma_1,\cdots ,N-\gamma_r$ .
 Since the determinant of a square matrix is the product of its eigenvalues it turns out that 
 \begin{equation}
 det(D^TD) = N^{n-r}\prod_{i=1}^{r}(N-\gamma_i)
 \end{equation}
 By analogy we have 
  $  CC^T -\lambda I_n =(n+d -\lambda)I_d -VV^T $
      Let   $\gamma = n+d -\lambda $ then we have the following equation:
  \begin{equation}
  \label{eigen1}
    CC^T -\lambda I_n = \gamma I_d -VV^T 
  \end{equation}
    Furthermore,  since   $rank(V) = r \leq d$,  the  $d\times d$ matrix  $VV^T$ has $r$  non-zero eigenvalues of $ \gamma_1,\cdots ,\gamma_r$  and $d-r$ zero eigenvalues. 

Therefore from Equation (\ref{eigen1}) the matrix $CC^T$   has $d-r$ eigenvalues $\lambda_i = n+d = N$; $ i = 1, \cdots , n-r$. The remaining $r$ eigenvalues of $D^TD$ are
   \\
    $N- \gamma_1,\cdots ,N-\gamma_r$ .
   \\
   It turns out that
   \begin{equation}
   \label{eigen2}
   det(CC^T) = N^{d-r}\prod_{i=1}^{r}(N-\gamma_i)
   \end{equation}
    By  Equations (\ref{eigen1}) and (\ref{eigen2}) we get 
    $\prod_{i=1}^{r}(N-\gamma_i) = N^{r-d} det(CC^T)$ which implies that
    
   $det(D^TD) = N^{n-r} N^{r-d} det(CC^T)  = N^{n-d} det(CC^T) $.
   \\
   It follows that  $|det(D)| = N^{\frac{n-d}{2}} |det(C)| $.
   \\
  To prove the second part of the theorem,   we have $|det(D)| = N^{\frac{n-d}{2}} |det(C)|$ . This  means that $D$ is invertible if and only if $C$ is invertible. Thus
 since $H_N$ is Hadamard we use the inversion formula for block matrices to get 
 \\
 $H_N^{-1} = \begin{bmatrix}
D & E \\
V & C
\end{bmatrix}^{-1} =  \begin{bmatrix}
[D- EC^{-1}V]^{-1} &-[D- EC^{-1}V]^{-1} EC^{-1} \\
-C^{-1}V[D- EC^{-1}V]^{-1}  & [C-VD^{-1}E]^{-1}
\end{bmatrix} = \frac{1}{N}H_N^T = \frac{1}{N}\begin{bmatrix}
D^T & V^T \\
E^T & C^T
\end{bmatrix}
  $
  \\
  $\frac{1}{N}D^T =[D- EC^{-1}V]^{-1} $
  \\
  $\frac{1}{N} C^T =  [C-VD^{-1}E]^{-1}    $
  \\
  The results follow easily.
\end{proof}
\begin{corollary}
\label{corol3}
Let $\Theta_n$ be the set of non-singular matrices of order n  with entries from $\lbrace -1, 1\rbrace $ for which the first column is $1_n$.  Let $D\in \Theta_n$. Then there exists a Hadamard matrix $H_N$ of the form 
$H_N = \begin{bmatrix}
D & E \\
V & C
\end{bmatrix}
$  such that  :

 $\boldsymbol{D^{-1}=\frac{1}{N}[D-EC^{-1}V]^T }$.
\end{corollary}
\begin{proof}
The $\lbrace -1, 1\rbrace$-matrix $M_n$ of order $2^n \times n$ formed with all the $n$-tuples  from the set $\lbrace -1, 1\rbrace$ can be extended to a Hadamard matrix $H_N$.  Let $\mathcal{C} =  m_1, \cdots, m_n$ be the set containing  the  columns of $M_n$. To construct $H_N$, 
it just suffices to  add the schur product of any non-empty subset of $\mathcal{C}$ as a new column to  $M_n$ as well as the column  vector $1_N$. It is not hard to see that since $D$ is non-singular of order $n$ its rows appear without repetition. Therefore each row of $D$ is also a row of $M_n$. It turns out that for any non-singular $\lbrace -1, 1\rbrace$-matrix $D$ there exists a Hadamard  matrix $H_N = \begin{bmatrix}
D & E \\
V & C
\end{bmatrix}
$.
By the last result stated in Theorem 1, $\boldsymbol{D^{-1}=\frac{1}{N}[D-EC^{-1}V]^T }$.
\end{proof}

\begin{lemma}
\label{contrast}
Consider a Hadamard matrix of the form $H_N = \begin{bmatrix}
D & E \\
V & C
\end{bmatrix}
$  where $D$ is invertible.  Assume the first columns of $H_N$ and $D$ are respectively $1_N$ and $1_n$ Then  we have :
\\
$\boldsymbol{(D-EC^{-1}V)^T1_n = \begin{bmatrix}
N\\
0_{n-1}
\end{bmatrix}} $

\end{lemma}

\begin{proof}

$(D- EC^{-1}V)^T1_n = D^T1_n - V^T(C^{-1})^TE^T1_n$
\\
Since we assume the first column of  $H_N$ is $1_N$ we have 
 $D^T1_n + V^T1_d = \begin{bmatrix}
N\\
0_{n-1}
\end{bmatrix}$ which  implies that  $D^T1_n =\begin{bmatrix}
N\\
0_{n-1}
\end{bmatrix} -V^T1_d $.
 
 Also $E^T1_n + C^T1_d = 0$ implies that $E^T1_n = - C^T1_d$ .
 
 It turns out that $(D-EC^{-1}V)^T1_n =\begin{bmatrix}
N\\
0_{n-1}
\end{bmatrix} -V^T1_d  +  V^T(C^{-1})^T C^T1_d$

$(D-EC^{-1}V)^T1_n = \begin{bmatrix}
N\\
0_{n-1}
\end{bmatrix} -V^T1_d  +  V^T(C^{-1}C)^T 1_d = -V^T1_d + V^T1_d  = 0$
 
 $(D-EC^{-1}V)^T1_n =\begin{bmatrix}
N\\
0_{n-1}
\end{bmatrix}$
\end{proof}
\begin{remark}

The   inverse of a non-singular $\lbrace -1, 1\rbrace$-matrix $D$ is always of the form  $D^{-1}=\frac{1}{N}(D-EC^{-1}V)^T$. 
 Lemma 1 shows that if one of the columns of  $D$   is $1_n$ then   $D^{-1}$ has a row for which the entries sum up to $1$. The entries of any of the remaining rows of $D^{-1}$ sum up to zero. This property is mathematically interesting but most importantly , it shows that in general  the estimator of any effect or interaction except the mean in a saturated design  is a contrast in terms of the runs. This eases the  interpretation of  the results of a saturated design conducted in a two-level factorial setup.

\end{remark}
\textbf{Deletion Algorithm}

In the context of a $2^k$-factorial experiment, suppose the experimenter has identified a subset of, say $n$, factorial effects and interactions which are supposed  to be non-negligible. Each of the remaining [$2^k - n$] effects and interactions is, by default, assumed to be negligible.
Furthermore suppose the  experimenter has minimal resource to carry out an $n$-run saturated design. Naturally the runs must be  chosen so that all the non-negligible effects are unbiasedly estimated. \\
The following steps seem to be plausible to follow :\\
\begin{enumerate}
\item  Write down the vector parameter of general mean, factorial effects and interactions as a column vector of dimension $2^k \times 1$ so that the general mean and all non-negligible parameters form a subset at the top and the negligible effects and interactions appear at the bottom. It is well-known that in a $2^k$-factorial set-up, we have standard representations for the effects and interactions in terms of the $2^k$ observations arising out of the full factorial experiment, if that were the case. \\

\item  Since the experiment is to be based on a suitably chosen subset of $n$ runs i.e., level-combinations of the factors, the experimenter has to make a judicious choice of these runs. Let $Y^{(1)}$ stand for the $n \times 1$ vector of observations realized after performing the experiment with suitably chosen $n$ runs. Let $Y^{(2)*}$ denote the complementary unobserved vector of dimension $(2^k - n) \times 1$   
in this context. \\

\item  The Hadamard matrix $H_N$ of order $N=2^k$ is decomposed in the usual manner wherein the matrix $D$ corresponds to the non-negligible parameters and the matrix $C$ corresponds to the negligible parameters. \\

\item  Let $\theta$ denote the vector parameter of the general mean and all the effects and interactions written in the style of (1) above. We use the notations $\theta^{(1)}$ and $\theta^{(2)}$ for the decompositions described under (1). Then $\theta^{(1)}$ represents the non-negligible effects and interactions whereas $\theta^{(2)}$ represents the 'zero' effects and interactions. \\

Note that $Y^{(1)}$ and $Y^{(2)}$ have already been defined as per this decomposition. Recall the partitioned matrix representation of $H_N$. From Yates' representation of the factorial effects and interactions, it follows that, but for a constant multiplier,  \\

\item  $\mathbb{E}[\frac{1}{N}H_N^TY]=\theta$ i.e., $(i) \frac{1}{N}\mathbb{E}[D^TY^{(1)} + V^TY^{(2)*}]=\theta^{(1)}$; (ii) $\mathbb{E}[E^TY^{(1)} + C^TY^{(2)*}]=0. $\\

\item  From (5)(ii) if $D$ is non-singular then $C$ is also non-singular by Theorem (\ref{theorem1})  and the best linear unbiased predictor [BLUP] of $Y^{(2)*}$ is given by :
 $
  \hat{Y}^{(2)*}  = -(C^{(-1)})^TE^TY^{(1)}
$. \\

\item Therefore, the BLUE of $\theta^{(1)}$ is given by $
\hat{\theta}^{(1)}=\frac{1}{N}(D- EC^{(-1)}V)^TY^{(1)}.
$ 

\item  Further,  the dispersion matrix of $var(\hat{\theta}^{(1)}) \propto [D - VC^{(-1)}E]^T[D - VC^{(-1)}E]$. \\

\item Incidentally, for d-optimal choice of the matrix $D$, or equivalently, of the matrix $C$, we need to minimize  $|det[(D - VC^{(-1)}E)|$ which is equivalent to maximizing $|det(C)|$.  \\
\end{enumerate}
Thus far we have explained the general principle underlying choice of the $D$-matrix for any given vector $\theta^{(2)}$ of negligible parameters. A little reflection suggests that the choice of the matrices (i) $\begin{bmatrix} D & E\end{bmatrix}$ of order $n \times N$ and (ii) $\begin{bmatrix} V& C \end{bmatrix}$ of order $N-n \times N$ are dictated by the two sets of non-negligible and negligible parameters. However, their splitting can be very much arbitrary subject to fulfilling the non-singularity conditions by the square matrices $D$ and $C$. We give three simple illustrations below. \\

Illustrations : $2^3$ - factorial experiment \\
(i) Only one effect / interaction is negligible: It may be seen that any one run can be deleted. \\
(ii) A pair of effects/interactions are negligible : It is not true that any two runs can be deleted. For example, if $F_{23}$ and $F_{123}$ are negligible, we may only delete the following pairs of runs : \\

$$
{(000,100),(000,110),(000,101),(000,111),(100,010), (100,001),(100,011),(010,110)} \\
$$
$$
{(010,101),(010,111), (110,001),(110,011),(001,101),(001,111),(101,011),(011,111)}.
$$

It is interesting to note that the runs $(011)$ and $(111)$ may as well be deleted. \\

(iii) When only the mean effect and three main effects are non-negligible, we may delete any of the following sets of four runs :
$$
{(000,100,010,001);(000,100,110,111);(000,010,110,111);(000,001,101,111)}
$$

Interestingly, the runs $(110), (101), (011), (111)$ do not form an admissible set for deletion. 

\section{Maximization criterion - an optimality consideration}

\subsection{A general rule for admissible set}
Theorem (\ref{theorem1}) states that whenever  a Hadamard matrix  $H_N$ is partitioned into block matrices  as  $\begin{bmatrix}
D & E \\
V & C
\end{bmatrix}$  then $|det(D)| = N^{\frac{n-d}{2}}|det(C)|$ and that the  matrix $D$ is non-singular if and only if the matrix $C$ is non-singular. The take-away message here in terms of a saturated design  is that there is direct relationship between a saturated design matrix $D$ and the matrix $C$. What it means is that when the experimenter is deciding which runs to keep in a saturated  design he has two equivalent options he can choose from. The first  choice is to  directly search  the runs that will make the matrix $D$ underlying the parameters of interest non-singular. The second choice is to select a set of runs so that the matrix $C$ underlying the negligible parameters is non-singular. The complement of such set  of runs is then the desired saturated design. Thus when the experimenter is trying to search for a saturated design  matrix $D$ of order $n$ then if $d = 2^k - n$  is such that  $n > d$,  it is advantageous to the experimenter to deal with a less complex problem by searching  for the matrix $C$ of order $d$. The design matrix can then be taken for granted by just taking the complement of $C$ which is the matrix $D$ in the partition of the Hadamard matrix $H_N$ above.  Furthermore since  $|det(D)| \propto |det(C)|$ it means that the Fisher information contained in a saturated design  $D$ is proportional to the amount of information contained in its complement matrix $C$. This is important if one desires to  find a d-optimal design  or classify saturated designs by their amount of Fisher information. In fact if $n> d$ the classification complexity of saturated design matrices $Ds$ boils down to the study of the determinant of the matrices $Cs$ which is less complex. 
  For convenience we make the following definitions:
\begin{defn}
Consider a $2^k$-factorial experiment where the full vector parameter $\theta$ is partitioned as $\theta = \begin{bmatrix} \theta^{(1)} &\theta^{(2)}\end{bmatrix}^T$ where $\theta^{(1)}$ is unknown and non-negligible and $\theta^{(2)}$ is negligible with cardinality $| \theta^{(1)} | = n$ and $| \theta^{(2)} | = 2^k - n$ . Furthermore suppose the set of runs $R$ is partitioned as $R = \begin{bmatrix} R_1 & R_2\end{bmatrix}$ with cardinality $|R_1| = n$ and $|R_2| = 2^k-n$ such that the Hadamard matrix $H_N$ underlying $\theta$ and $R$ is written as :
\\
\begin{tabular}{c|| c c}
      & $\theta^{(1)}$& $\theta^{(2)}$\\
      \hline\hline
$R_1$ & $D$ & $E$\\
$R_2$ & $V$ & $C$
\end{tabular} where $D$ and $C$ are square matrices of order $n$ and $2^k-n$ respectively. 
We make the following definitions :
\begin{enumerate}
\item  We shall say the set $R_2$ is an admissible set for deletion if the matrix $C$ is non-singular.
\item If $R_2$ is admissible then its complement $R_1$ is a saturated design for the estimation of $\theta^{(1)}$.
\item If $|det(C)|$ is maximal we shall say $R_2$ is a d-optimal admissible set for deletion
\item  If $|det(C)|$ is maximal then $R_2$ is a d-optimal admissible set and we shall say that its complement $R_1$ is a d-optimal design for the estimation of $\theta^{(1)}$.
\end{enumerate}
\end{defn}






\subsection{Example of classification of saturated designs by determinant for mean, main effects and the two-factor interactions:  $2^4$-experiment case}
\label{example}
Suppose for one reason or the other the experimenter is interested in  a saturated design with $k=4$  factors where the important effects are  $F_0 , F_1 , F_2, F_3 , F_4 , F_{12}, F_{13}, F_{14}, F_{23}, F_{24}, F_{34}$ and  the negligible effects are $ F_{123}, F_{124},  F_{134}, F_{234},  F_{1234} $.  For such a problem the experimenter needs to find 11 runs out of the 16 possible runs that would make the underlying design matrix of order 11 non-singular. It turns out that the task of finding such a design matrix is computationally more involved than finding an admissible set for the 5 negligible parameters. Since there are only 5 negligible effects, a shortcut to finding the desired  design matrix would be to first take a look at the $2^4$-full factorial design matrix which is a Hadamard matrix and try to come up with a non-singular design matrix for the negligible effects $F_{123}, F_{124},  F_{134}, F_{234},  F_{1234} $.  The complement of such design matrix would be the desired design matrix to estimate the important effects and interactions. 
 
In the  first  $2^4$-Hadamard matrix in Figure (\ref{hadamard4}) , we display in bold the $C$- matrix $C_1$ underlying the runs\\
 $\lbrace 1101, 0011, 1011, 0111, 1111 \rbrace $  and the interactions $\lbrace F_{123}, F_{124}, F_{134}, F_{234}, F_{1234} \rbrace $ . \\
$C_1 = \begin{bmatrix}
-1&1&-1&-1&-1\\
1&1&-1&-1&1\\
-1&-1&1&-1&-1\\
-1&-1&-1&1&-1\\
1&1&1&1&1
\end{bmatrix}$ .
\\
It is not hard to see that this particular $C$-matrix is singular since the first and the last columns are the same. Thus the set of runs $\lbrace 1101, 0011, 1011, 0111, 1111 \rbrace $ is not an admissible set  and so its complement set of runs 
$$
\lbrace 0000, 1000, 1100, 0010, 1010, 0110, 1110, 0001, 1001, 0101 \rbrace 
$$ 
is not a valid set for  the estimation of the mean, the main effects and the two-factor interactions. In fact by Theorem (\ref{theorem1}) the $D$-matrix underlying this set is also singular.

On the other hand, the $C$-matrix $C_2$ displayed in bold in the second $2^4$-Hadamard matrix in Figure (\ref{hadamard4})  underlies the runs 
\\ $\lbrace 0000, 1100, 1010,  1001, 1111 \rbrace $  and the interactions $\lbrace F_{123}, F_{124}, F_{134}, F_{234}, F_{1234} \rbrace $ . \\

$C_2 = \begin{bmatrix}
-1&-1&-1&-1&1\\
-1&-1&1&1&1\\
-1&1&-1&1&1\\
1&-1&-1&1&1\\
-1&-1&-1&1&-1
\end{bmatrix}$ .
\\
The absolute value of the determinant of this choice of $C$-matrix is $|det(C_2)| = 48$.  Therefore, the set of runs $\lbrace 0000, 1100, 1010,  1001, 1111 \rbrace $ is an admissible set for deletion. \\

It is further observed that the choice of the matrix $C$ leading to an admissible set for deletion,  is not necessarily unique.  In that case, we might like to make a judicial choice for it. For this we refer to the dispersion matrix of the blue of $\theta^{(1)}$ and minimize the generalized variance. 

It is readily seen that it amounts to maximization of  the absolute value of $det(C)$ where $C$ is a square matrix of dimension $N-n$ with elements  in $\lbrace -1, 1\rbrace$. \\

It is a well known result that the maximal  absolute value of the determinant of $\lbrace -1, +1\rbrace$-matrices of order 5 is $48$ and It is easy to verify that computationally.

Thus the absolute value of the corresponding  $D$-matrix $D_2$ is $|det(D_2)| = 16^{\frac{11-5}{2}}\times 48$.
We may thus conclude that the set \\
 $$
\lbrace 1000, 0100,  0010, 0110,  1110, 0001,  0101, 1101, 0011, 1011, 1111 \rbrace 
$$ 
is  d-optimal design for the estimation of $F_0 , F_1 , F_2, F_3 , F_4 , F_{12}, F_{13}, F_{14}, F_{23}, F_{24}, F_{34}$.

\begin{center}
\begin{figure}
\includegraphics[trim= {0, 10cm, 0, 0} , clip, scale =0.7]{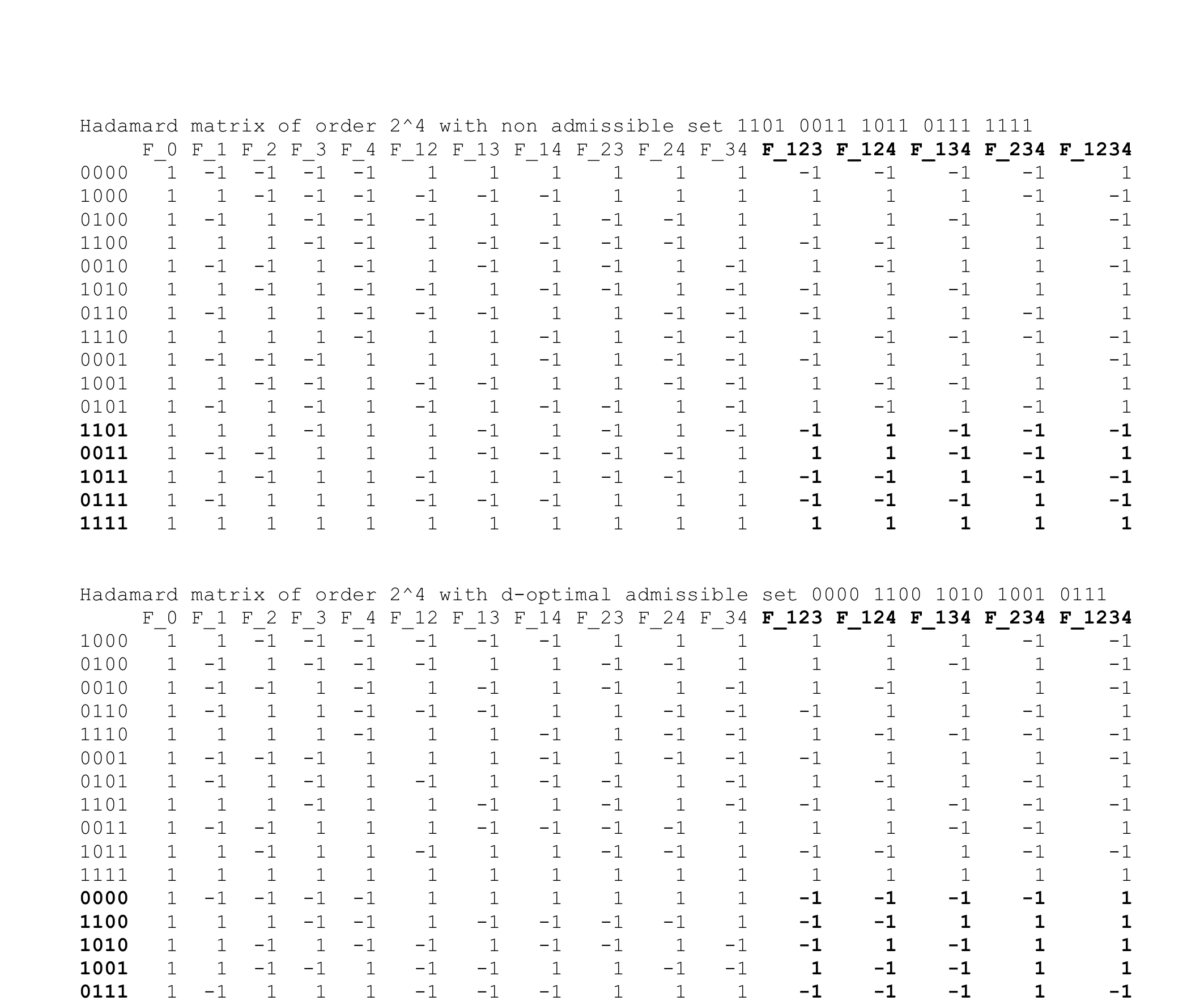}
\caption{Hadamard matrix of order $2^4$}
  \label{hadamard4}
  \end{figure}
\end{center}
\begin{center}
\begin{figure}

\includegraphics[trim= {0, 12cm, 0, 0} , clip, scale =0.7]{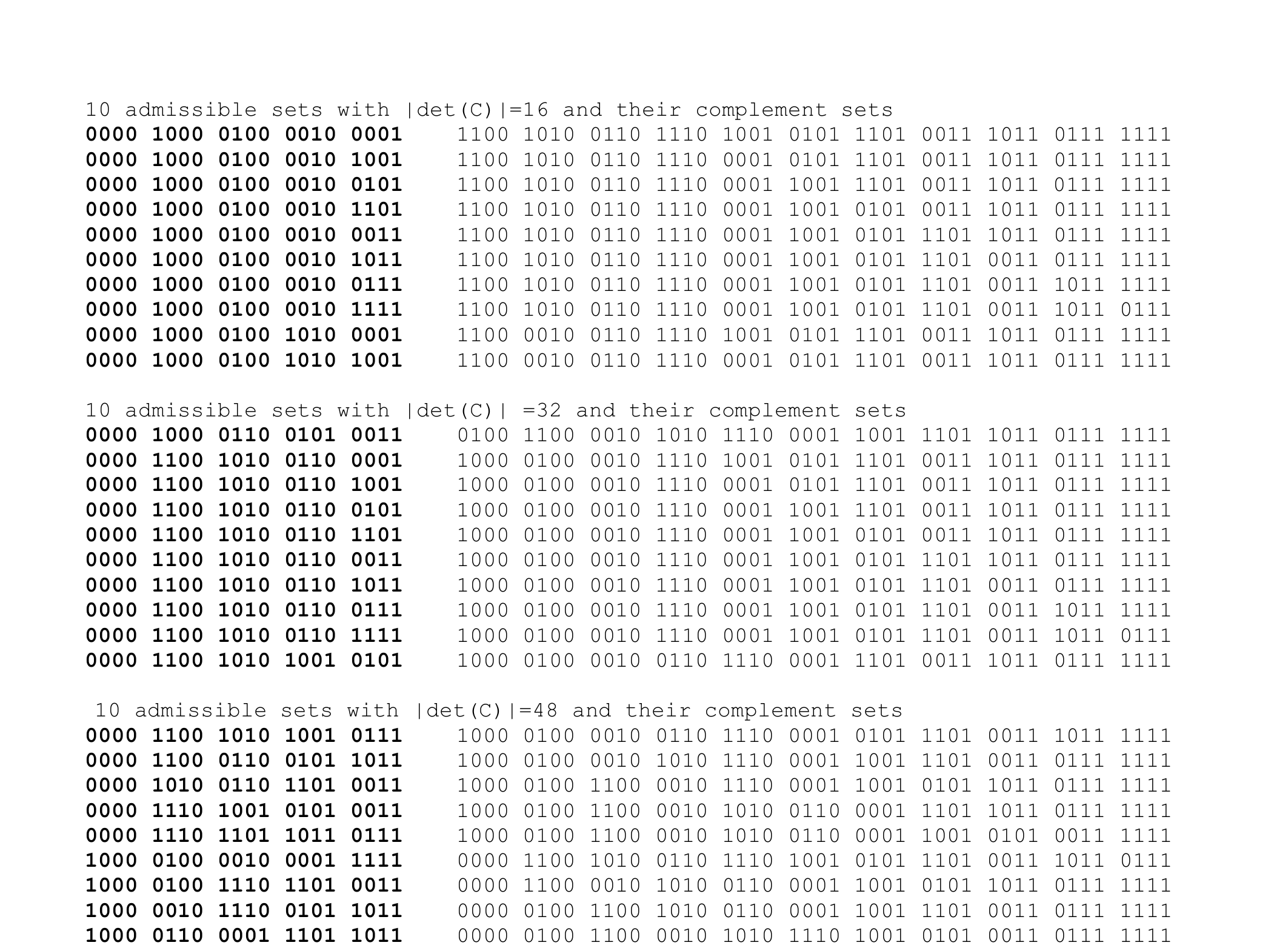}
\caption{Classification of  saturated designs by determinant}
\label{class}
\end{figure}
\end{center}

The spectrum  $S_n$ of the determinant function of  $\lbrace -1, +1\rbrace$-matrices $M_n$ of order $n$ is defined to be the set of values taken by $|det(M_n)|/2^{n-1}$.  It is well known in the literature that the spectra $S_5 = \lbrace 0, 1, 2, 3\rbrace$ and\\ 
$S_{11}= \lbrace 0,\cdots, 268 ;\quad 270,\cdots , 276; \quad 278,\cdots 280; \quad 282,\cdots, 286 ;\\
 288, 291; \quad 294, \cdots, 297; \quad 304, 312, 315, 320\rbrace$ . See Orrick [\cite{orrick1}]. It is important to point out that the absolute value of the determinant of $\lbrace -1, +1\rbrace$-matrices  of order $11$ can take up to $296$ values. Moreover  the absolute value of the determinant of $\lbrace -1, +1\rbrace$-matrices  of order $5$ can take up to $4$ values including the value $0$.  Thus for the problem of classifying the saturated designs for $F_0 , F_1 , F_2, F_3 , F_4 , F_{12}, F_{13}, F_{14}, F_{23}, F_{24}, F_{34}$ by determinant  there will be only 3 classes of designs. This is because $S_5$ can only take $3$ non-zero values.  In fact we have  $|det(M_5)|$ takes value from  $ \lbrace 0, 16, 32, 48\rbrace$.  Therefore by Theorem (\ref{theorem1}),  $|det(C)|$ takes values from $\lbrace 0, 16, 32, 48\rbrace$ and $|det(D)|$ takes values from \\ $\lbrace 0, 16^{3}\times 16, 16^{3}\times 32, 16^{3}\times 48\rbrace$ . In Figure (\ref{class}) we give 10 examples in each class of saturated designs. the last class correspond to the d-optimal design class. 

\newpage
\section{Concluding Remarks}
The construction of saturated designs for two level factorial experiment has gained a substantial  interest over a long period of time. Numerous papers have been written about the classification of saturated design matrices of fixed order  via the spectrum of the determinant function.  Thus the spectra of the determinant function $S_n$ for $\lbrace -1, +1\rbrace$-matrices  of order $n$ are  well known in the literature for order up to $11$. The spectrum of order $n =8$ is due to Metropolis, Stein and Well [\cite{metropolis}]. For $n = 9$ and $n=10$,  the spectra were computed by \v{Z}ivkovi\'c [\cite{zivkovic}] and the spectrum for $n = 11$ is due to Orrick [\cite{orrick1}].
Furthermore  many other papers have studied d-optimal saturated design matrices for a  fix order.  Orrick [\cite{orrick1}] constructed a  d-optimal design matrix of order $15$.  T. Chadjipantelis, S. Kounias and C. Moyssiadis [\cite{chad}] came up with a d-optimal design of order $21$. The the d-optimal design matrix discussed by these papers is the $\lbrace -1, +1 \rbrace $-matrix of order $n$ with the largest absolute value of the determinant among all other $\lbrace -1, +1 \rbrace $-matrices of order $n$. In some sense these d-optimal design matrices are the global d-optimal design matrices.  From a statistical design perspective,  these matrices are really d-optimal if the only parameters that are important are the mean and the main effects. However when the experimenter is willing to construct a d-optimal design matrix that includes the mean, main effect and a certain number of interactions, the columns of the interactions in the design matrix are obtained by the Sch\"ur product of the appropriate columns of main effects. With that being said the columns of the interactions are deterministic once the columns of the main effects are chosen. It turns out that because of the restriction imposed by the interaction columns, a d-optimal design matrix of order $n$ for a chosen vector parameter that includes the mean , the main effects and a selected set of interaction may not be the global d-optimal design of order $n$. In many cases Theorem (\ref{theorem1}) can simplify the problem of studying the spectrum or finding a d-optimal design that includes the mean, the main effects and a selected set of interaction. The example given in subsection (\ref{example}) is a nice illustration.


\section{Acknowledgements}
This work is partially supported by the US National Science Foundation(NSF Grant 1809681).


\bibliography{uictman}

\end{document}